\documentclass[a4paper, 11 pt, twoside]{amsart}
\usepackage{srollens-en}[2011/03/03]
\usepackage[left = 3.5cm, right = 3.5cm, headsep = 6mm,
footskip = 10mm, top = 35mm, bottom = 35mm, footnotesep=5mm, headheight =
2cm]{geometry}

\usepackage{tikz, tikz-cd}

\usepackage[unicode,bookmarks, pdftex, backref = page]{hyperref}

\DeclareMathOperator{\Gl}{Gl}

\newcommand{\Wedge}{\textstyle\bigwedge}

\renewcommand{\lg}{\ensuremath{\mathfrak g}}
\newcommand{\lh}{\mathfrak h}
\newcommand{\einsnull}[1]{{{#1}^{1,0}}}
\newcommand{\nulleins}[1]{{{#1}^{0,1}}}
\newcommand{\dq}{\delbar}

\newcommand{\I}{\mathrm{i}}

\title[Dolbeault cohomology of foliated nilmanifolds]{Dolbeault cohomology of complex  nilmanifolds foliated in toroidal groups}

\author{Anna Fino}
\address{Anna Fino\\Dipartimento di Matematica \lq \lq G. Peano\rq \rq  \\
Universit\'a di torino\\
via Carlo Alberto 10\\
10123 Torino\\
Italy}

\email{annamaria.fino@unito.it}

\author{S\"onke Rollenske}
\address{S\"onke Rollenske\\FB 12/Mathematik und Informatik\\
Philipps-Universit\"at Marburg\\
Hans-Meerwein-Str. 6\\
35032 Marburg\\
Germany}
\email{rollenske@mathematik.uni-marburg.de}

\author{Jean Ruppenthal}
\address{Jean Ruppenthal, Department of Mathematics, University of Wuppertal, Gau{\ss}str. 20, 42119 Wuppertal, Germany.}
\email{ruppenthal@uni-wuppertal.de}

\subjclass[2010]{Primary 22E25; Secondary  37F75, 53C30, 53C15}

\begin{document}

\begin{abstract}
It is conjectured that the Dolbeault cohomology of a complex nilmanifold $X$  is computed by left-invariant forms.  
We prove this under the  assumption that $X$ is suitably foliated in toroidal groups and deduce that the conjecture holds in real dimension up to six. 

Our approach generalises previous methods, where the existence of a holomorphic fibration was a crucial ingredient.

\end{abstract}

\maketitle
\setcounter{tocdepth}{1}
\tableofcontents

\section{Introduction}
 A nilmanifold is a compact quotient  $M = \Gamma \backslash G$, where $G$  is a real simply-connected nilpotent Lie group and $\Gamma$  is a discrete cocompact subgroup of G. 
 By  \cite{nomizu54} the inclusion of left-invariant differential forms in the de Rham complex $\Wedge^* \lg^* \rightarrow \ka^*( M)$ induces an isomorphism between the Lie algebra cohomology of $\lg$ and the de Rham
cohomology of $M$, where $\lg$ is the Lie algebra of $G$.
  
 By a complex nilmanifold  $(M, J)$ we will mean a nilmanifold $M = \Gamma \backslash G$ endowed with a complex structure  $J$ 
 induced by an integrable left-invariant complex structure on $G$  and
 thus uniquely determined by an (integrable) complex structure  $J$ on the Lie algebra $\lg$. 
The study of the Dolbeault cohomology of  complex nilmanifolds  is motivated by the fact   that nilmanifolds  provide examples of symplectic manifolds with no K\"ahler structure. Indeed, Benson and Gordon \cite{ben-gor88} proved that a complex nilmanifold is K\"ahler if and only if it is a torus (see also \cite{hasegawa89}).

By \cite{con-fin01, rollenske09a}  there exists a natural map  
\[\iota:   H^{*,*}   (\lg, J)   \rightarrow  H^{*,*} (M, J)\]   
from the Dolbeault cohomology as computed by left-invariant forms to the usual Dolbeault cohomology, that is always injective. The map $\iota$ is known to be an isomorphism in many cases:
if  $(M, J)$  is a complex  nilmanifold  such that   $\lg$  admits  either  a (principal) torus bundle series with respect to $J$ (\cite{con-fin01, cfgu00}) or  a    stable torus bundle series  \cite{rollenske10},  then
the isomorphism  $H^{*,*}   (\lg, J)   \cong  H^{*,*} (M, J)$  holds. This class includes all complex parallelisable nilmanifolds.  
In addition the property of $\iota$ being an isomorphism is open in the space of left-invariant complex structures \cite{con-fin01} and holds after suitably enlarging the lattice \cite[Prop. 3.16]{rollenske10}. We are thus lead to conjecture that $\iota$ is an isomorphism on every complex nilmanifold. 
 
So far, only in real dimension four the conjecture had been verified, rather trivially, because there is only the Kodaira-Thurston manifold to consider. Using Salamons classification of six-dimensional nilpotent Lie algebras admitting a complex structure \cite{salamon01} the conjecture was verified in dimension six in \cite{rollenske09b} with the exception of $\lh_7$  with structure equations
\[
\left \{ \begin{array}{l} d e^j =0, j = 1,2,3,\\[2pt]
d e^4 = e^1 \wedge e^2,\\[2pt]
d e^5 = e^1 \wedge e^3,\\[2 pt]
d e^6 = e^2 \wedge e^3.
\end{array}
\right.
\]
The problem in this case is that for a general lattice and complex structure on $\lh_7$, the corresponding nilmanifold does not admit any holomorphic fibration onto a complex nilmanifold of lower dimension.

This paper grew out of the attempt to replace the role of a holomorphic fibration played in the previous proofs by holomorphic foliations. 
We show that in some cases the Dolbeault cohomology of a complex nilmanifold  $(\Gamma \backslash G, J)$  can be be computed even if the nilmanifold is not fibred but  foliated  in toroidal groups (Theorem \ref{thm: toroidal foliation}). As an application we prove that 
 the conjecture holds  also for a complex nilmanifold  $(\Gamma \backslash G, J)$  with $\lg \isom \lh_7$ and thus for every  6-dimensional  complex nilmanifold.

\subsubsection*{Acknowledgements}
The first author is supported   by GNSAGA of INdAM. The second and the third author are grateful for support of the DFG through the Emmy Noether program. The second author is also partially supported by the DFG through SFB 701. He would like to thank numerous people for discussions relating to this conjecture over the years, including Daniele Angella and Dan Popovici.

In the final stages of this project we have become aware that Adriano Tomassini and  Xu Wang are considering the same question from a more analytical view point by studying Laplacians with respect to a fibration.

\section{Preliminaries}

\subsection{Reminder on Dolbeault cohomology}\label{reminder}
Let $M$ be a differentiable manifold endowed with a complex structure $J.$ Recall that a  complex structure on a differentiable manifold $M$ is a vector bundle endomorphism $J$ of the tangent bundle which satisfies $J^2=-\id$ and  the integrability condition 
\begin{equation}\label{nijenhuis}
 [x,y]-[Jx,Jy]+J[Jx,y]+J[x,Jy]=0,
\end{equation}
for every pair of vector fields $x$ and $y$ on $M$. The endomorphism $J$ induces a decomposition of the complexified tangent bundle by letting pointwise $\einsnull TM\subset T_\IC M=T M\tensor\IC$ be the $\I$-eigenspace of $J$. Then the $-\I$-eigenspace is $\nulleins TM=\overline{\einsnull TM}$. Note that $\einsnull TM $ is naturally isomorphic to $(TM,J)$ as a complex vector bundle via the projection, and the integrability condition can be formulated as $[\einsnull TM, \einsnull TM]\subset \einsnull TM$.

The bundle of differential $k$-forms decomposes in the following way
\[\Wedge^kT^*_\IC M=\bigoplus_{p+q=k}\Wedge^p \einsnull{T^*}M\tensor \Wedge^q\nulleins{T^*}M=\bigoplus_{p+q=k}\Wedge^{p,q}T^*M,\]
and we denote by $\ka^{p,q}(M)$ the $\kc^\infty$-sections of the bundle $\Wedge^{p,q}T^*M$, i.e., the global differential forms of type $(p,q)$.

The integrability condition \eqref{nijenhuis} is equivalent to the decomposition of the differential $d=\del+\delbar$  and for all $p$ we get  the Dolbeault complex
\[(\ka^{p,\bullet}(M_J), \delbar): 0\to \ka^{p,0}(M)\overset\delbar\longrightarrow \ka^{p,1}(M)\overset\delbar\longrightarrow \dots\]
The Dolbeault cohomology groups $H^{p,q}(M)=H^q(\ka^{p,\bullet}(M), \delbar)$ are one of the most fundamental holomorphic invariants of $(M, J)$. From another point of view, the  Dolbeault complex computes the  cohomology groups of  the sheaf $\Omega^p_{M}$ of holomorphic $p$-forms.

In case $(M, J)$ is a  complex nilmanifold all of the above can be considered at the level of left-invariant forms. Decomposing $\lg^*_\IC=\einsnull{\lg^*}\oplus \nulleins{\lg^*}$ and setting $\Wedge^{p,q}\lg^*=\Wedge^p\einsnull{\lg^*}\tensor \Wedge^q\nulleins{\lg^*}$ we get subcomplexes
\begin{equation}\label{inclusion}
 (\Wedge^{p,\bullet}\lg^*, \delbar)\into (\ka^{p,\bullet}(M, J), \delbar). 
\end{equation}
In fact, the left hand side has a purely algebraic interpretation worked out in \cite{rollenske09a}: $\nulleins\lg$ is a Lie subalgebra of $\lg_\IC$ and the adjoint action followed by the projection to the $(1,0)$-part makes $\einsnull\lg$ into an $\nulleins\lg$-module. Then the complex $(\Wedge^{p,\bullet}\lg^*, \delbar)$ computes the Lie algebra cohomology of $\nulleins\lg$ with values in $\Wedge^p\einsnull{\lg^*}$ and we call
\[H^{p,q}(\lg, J)=H^q(\nulleins\lg, \Wedge^p\einsnull{\lg^*})=H^q(\Wedge^{p,\bullet}\lg^*, \delbar)\]
the Lie algebra Dolbeault cohomology of $(\lg, J)$. 

We can now formulate the analogue of Nomizu's theorem for Dolbeault cohomology as a conjecture.
\begin{custom}[Conjecture]\label{conj}
 Let $(M, J)$ be a nilmanifold with left-invariant complex structure. Then the map
\begin{equation*}
\phi_J: H^{p,q}(\lg, J)\to H^{p,q}(M,J)
\end{equation*}
induced by \eqref{inclusion} is an isomorphism.
\end{custom}
It is known that $\phi_J$ is always injective (see \cite{con-fin01} or \cite{rollenske09a}).

\subsection{Hausdorff quotients}

As we have to deal with cohomology groups which are not necessarily Hausdorff,
we make use of the following notation:

\begin{defin}
By a {\it quotient of a Fr{\'e}chet space}, in short, QF-space, we understand the quotient space
$E=F/A$ of a Fr{\'e}chet space $F$ by an arbitrary subspace $A\subset F$, carrying the induced
quotient topology. The Hausdorff quotient of $E$ is then $\overline{E}:= E/\overline{\{0\}} = F/\overline{A}$,
where $\overline{\{0\}}$ is the smallest closed subspace of $E$ and $\bar A$ is the closure of $A$.
\end{defin}

Note that a QF-space is a vector space with topology, but in general not a topological vector space,
because the definition of a topological vector space requires that points (particularly $\{0\}$) are closed.
Let us recall some simple but crucial observations about QF-spaces.

\begin{lem}\label{lem:QF1}
Let $E=F/A$ be a QF-space. Then the following are equivalent:
\begin{enumerate}
\item[(a)] $A$ is closed.
\item[(b)] $E$ is a Fr{\'e}chet space.
\item[(c)] $E$ is Hausdorff.
\item[(d)] $\{0\}$ is closed in $E$.
\end{enumerate}
\end{lem}

\begin{proof}
$(a) \Rightarrow (b)$ It is well-known that $E=F/A$ is again a Fr{\'e}chet space if $A$ is closed.
$(b) \Rightarrow (c)$ A Fr{\'e}chet space is clearly Hausdorff.
$(c) \Rightarrow (d)$ In a Hausdorff space points are closed.
$(d) \Rightarrow (a)$ The projection $\pi: F \rightarrow F/A$ is by definition continuous.
So, if $\{0\}$ is closed in $E$, then $A=\pi^{-1}(\{0\})$ is closed in $F$.
\end{proof}

 \begin{lem} \label{lem: Hausdorff quotient for group cohomology}
  Let $\Lambda$ be a discrete group acting continuously on a QF-space $V$. Assume there is a finite-dimensional $\Lambda$-submodule $V_0$ of $V$ such that the map $V_0 \to \overline{ V} = V/\bar 0$ to the Hausdorff quotient is an isomorphism. 
  Then $V = \bar 0\oplus V_0$ splits as a $\Lambda$-module and taking the Hausdorff quotient commutes with group cohomology.
 \end{lem}
\begin{proof}
Since the action of $\Lambda$ is continuous the subspace $\bar 0$ is a submodule and thus we have the claimed splitting. 
 
Consequently, we have $H^k(\Lambda, V) = H^k(\Lambda, \bar 0)\oplus H^k(\Lambda, V_0)$. The second summand is Hausdorff, because it is finite-dimensional. The Hausdorff quotient of the first summand is trivial, because it is a subquotient of a topological vector space for which this holds. 
\end{proof}

\section{Complements on abelian complex Lie groups}\label{sect: toroidal groups}

We recall some properties of abelian complex Lie groups from \cite{abe-kopferman01}.
The following class of groups is of particular importance:
\begin{defin}
An abelian complex Lie group $F$ is called a toroidal group if $H^0(F, \ko_F)=\IC$, i.e.,   if all holomorphic functions are constant.
\end{defin}

Let $F$ be a connected abelian complex Lie group.
By \cite[1.1.2]{abe-kopferman01} there exists a discrete subgroup $\Gamma$ in the  complex vector space $\gothf = T_0 F$ such that $F= \gothf/\Gamma$. 

Let $\gothf_0$ be the largest complex subspace contained in the real subspace of $\gothf$  spanned by $\Gamma$. Then by \cite[Sect. 1]{abe-kopferman01} one can choose a basis for  $\gothf$ such that $\gothf_0$ is spanned by the last $q$ coordinate vectors and $\Gamma$ is spanned by the columns of a block matrix of the form  
\begin{equation}\label{eq: period matrix}
\begin{pmatrix} 
0 & 0 &  0 \\
I_b & 0 & 0\\
0 & I_{n-q} & R\\
0 & 0& P\end{pmatrix}
\end{equation}
where $P\in \mathrm{Mat}(q, 2q, \IC)$ is the period matrix of a compact complex torus and $R \in \mathrm{Mat}(n-q, 2q, \IR)$ is a real matrix, called glueing matrix, satisfying the irrationality condition 
\[ \forall \sigma \in \IZ^{n-q}\colon \sigma^tR\not \in \IZ^{2q}.\]
Then 
\[ T = \IC^n / \begin{pmatrix} 
 I_{n-q} & R\\
 0& P\end{pmatrix}\IZ^{n+q}\]
is a toroidal group and $q$ is called the rank of $T$. The above choice of basis induces a $(\IC^*)^{n-q}$-principal bundle structure $T\to F_0 = \gothf_0/P\IZ^{2q}$ over a compact complex torus. 

By abuse of notation we call such a coordinate system or sometimes just the splitting $\gothf = \gothk\oplus \gothf_0$ induced by it a toroidal coordinate system for $F$. 

As a consequence we get: 
\begin{prop}[\protect{Remmert, Morimoto (\cite{MR0207893}, \cite[1.1.5]{abe-kopferman01})}]\label{prop: structure abelian cx Lie-groups}
 The toroidal coordinate system induces a biholomorphism $F\isom \IC^a \times (\IC^*)^b \times T$, and the factors of this decomposition are uniquely determined. 
\end{prop}

In the toroidal case there is an important distinction found by Vogt \cite{vogt82} that can be detected by Dolbeault cohomology.
\begin{thm}\protect{\cite[Sect. 2.2]{abe-kopferman01}}  \label{thm: vogt} Let $F$ be a toroidal group and $\pi\colon F\to F_0$ the principal bundle structure induced by a choice of toroidal coordinates $\gothf = \gothk\oplus \gothf_0$ with period matrix as in \eqref{eq: period matrix}.

Then the following properties are equivalent:
\begin{enumerate}
\item The glueing matrix $R$ satisfies the condition
\[ \exists r\in \IR^{>0}\, \forall \sigma \in \IZ^{n-q}\setminus\{0\}\colon r^{-|\sigma|}\leq \text{dist}(R^t\sigma, \IZ^{2q}) = 
\inf_{\tau \in \IZ^{2q}} \left| R^t\sigma - \tau \right|.\]
 \item $H^{0,1}(F)$ is finite-dimensional.
 \item $H^{p,q}(F)$ is finite-dimensional for all $p,q$.
 \item For all $q$ the projection $\pi$ induces an isomorphism $\pi^*\colon H^{0,q}(F_0) \to H^{0,q}(F)$.
 \item For all $p,q$ we have
 \[ H^{p,q}(F) \isom \Wedge^p\gothf^*\tensor \Wedge^q\bar \gothf_0^*.\]
\item Every line bundle on $F$ is a theta line bundle (see \cite[Sect. 2.1]{abe-kopferman01} for this notion).
 \end{enumerate}
If one of these conditions is satisfied, then we call   $F$  a toroidal theta group.  Otherwise we call $F$ a toroidal wild group. 
\end{thm}

\begin{exam}\label{ex: theta and wild in dim two} 
Let us give examples of toroidal theta and wild groups applying the first equivalent condition.

Consider for $a\in \IR$ the two-dimensional abelian complex Lie-group $F_a$ with 
$\Gamma$ given in toroidal coordinates by
\[\begin{pmatrix}
   1& 0& a\\ 0 & 1& \I
  \end{pmatrix}
  \]
Then projection to the second coordinate induces a $\IC^*$-principal bundle over the elliptic curve $\IC/\IZ[\I]$ and by the irrationality condition $F_a$ is toroidal if and only if $a\not\in \IQ$.

Applying  condition \refenum{i} of Theorem \ref{thm: vogt} to this period matrix we deduce that  $F_a$ is a toroidal theta group if and only if there is a positive number $r\in \IR$ such that 
\begin{equation}\label{eq: theta condition dim 2}  {r^{-|\sigma|}}\inverse{|\sigma|}\leq \left| a - \frac\tau\sigma \right| \quad\text{for all $\frac{\tau}\sigma \in \IQ$}.
\end{equation}

By a Theorem of Liouville \cite{Liouville}, algebraic numbers cannot be approximated well by rational numbers, that is, for every algebraic number $a$ with $[\IQ(a):\IQ] = d\geq 2$ there is a positive constant $c$ such that 
\[  \frac{c}{|\sigma|^d}\leq \left| a - \frac\tau\sigma \right| \quad\text{for all $\frac{\tau}\sigma \in \IQ$}.\]
From this it is easy to see that for every algebraic number $a$ the  inequality \eqref{eq: theta condition dim 2}  can be satisfied for  some constant $r$ and thus $F_a$ is a toroidal theta group.

On the other hand, choosing for $a$ a transcendental number which is very well approximated by rational numbers such as 
\[a = \sum_{k = 1}^\infty \frac1{10 ^{10 ^{k!}}}\]
then one can easily see that \eqref{eq: theta condition dim 2} cannot be satisfied for any $r$. 
\end{exam}

\begin{prop}\label{prop: Hausdorff quotients of Dolbeault cohomology for abelian Lie groups}
 Let $F = \gothf/\Gamma$ be a toroidal group. 
 Choose a toroidal coordinate system inducing a splitting $\gothf = \gothk\oplus \gothf_0$. 
 Then the inclusion $\Wedge^p\gothf^*\tensor \Wedge^q\bar \gothf_0^*\subset H^{p,q}(F)$ induces a splitting 
  \begin{equation}\label{eq: Dolbeault decomposition}
 H^{p,q}(F) = \left(\Wedge^p\gothf^*\tensor \Wedge^q\bar \gothf_0^*\right)\oplus \bar 0,  
  \end{equation}
  that is $\Wedge^p\gothf^*\tensor \Wedge^q\bar \gothf_0^*$ is naturally isomorphic to the Hausdorff quotient of the Dolbeault cohmology group $H^{p,q}(F)$. Here, $H^{p,q}(F)$ carries the usual topology as the quotient of a Fr\'echet space of smooth forms.
\end{prop}

\begin{rem}
Proposition
\ref{prop: Hausdorff quotients of Dolbeault cohomology for abelian Lie groups} can be generalized to abelian complex Lie groups.
In that case, the factor $\IC^a \times (\IC^*)^b$
(see Proposition \ref{prop: structure abelian cx Lie-groups})
has to be incorporated by means of the K\"unneth formula.
We do here, however, only need the statement for toroidal groups.
\end{rem}

\begin{proof}
In the toroidal theta case this is just Vogt's Theorem (see Theorem \ref{thm: vogt} above).

The toroidal wild case follows from the proof of Vogt's Theorem. We will explain that now.
Let $Z^{p,q}(F) = \ker \dq \subset \ka^{p,q}(F)$ and $B^{p,q}(F) = \im \dq \subset \ka^{p,q}(F)$,
where $\ka^{p,q}(F)$ carries the usual structure as the Fr\'echet space of smooth forms.
Hence, $H^{p,q}(F) = Z^{p,q}(F) / B^{p,q}(F)$ by definition, and \eqref{eq: Dolbeault decomposition}
is equivalent to
\begin{eqnarray}\label{eq:quotient01}
\frac{Z^{p,q}(F)}{\overline{B^{p,q}(F)}} = \Wedge^p\gothf^*\tensor \Wedge^q\bar \gothf_0^*.
\end{eqnarray}
To this end, we show the following:

\medskip
{\bf Claim:} {\it Modulo $\overline{B^{0,q}(F)}$ every form $\omega\in Z^{0,q}(F)$ is cohomologous 
to a form with constant coefficients (in toroidal coordinates).}
 
 \medskip
 To show the claim, we recall some main steps of the proof of \cite[Prop. 2.2.5]{abe-kopferman01}. 
 We only need to consider the case $q>0$ as holomorphic functions are constant on toroidal groups.
 So, let
 $\omega$ be a $\dq$-closed $\Gamma$-periodic $(0,q)$-form on $\gothf\cong \IC^n$, $q>0$.
 Then 
 \begin{eqnarray}\label{eq:quotient02}
 \omega &=& \sum_{\sigma\in \IZ^{n+q}} \omega^{(\sigma)},
 \end{eqnarray}
 where $\omega^{(0)}$ is a $(0,q)$-form with constant coefficients in toroidal coordinates
 and the sum, coming from a Fourier expansion, converges in the $\ka^{p,q}$-topology
 (see also \cite[Prop. 2.2.7]{abe-kopferman01}). The $\omega^{(\sigma)}$ are $\Gamma$-periodic. 
 But
 \begin{eqnarray*}
 \omega^{(\sigma)} &=& \dq \eta^{(\sigma)}
 \end{eqnarray*}
 with smooth $\Gamma$-periodic $(0,q-1)$-forms $\eta^{(\sigma)}$.
 Thus
 \begin{eqnarray*}
 \omega - \omega^{(0)} &=& \lim_{N\rightarrow \infty} \sum_{\substack{0\neq \sigma\in \IZ^{n+q}\\ |\sigma| \leq N}} \omega^{(\sigma)}
 = \lim_{N\rightarrow \infty} \dq \sum_{\substack{0\neq \sigma\in \IZ^{n+q}\\ |\sigma| \leq N}} \eta^{(\sigma)}
 \end{eqnarray*} 
 This statement, contained in the proof of \cite[Prop. 2.2.5]{abe-kopferman01}, proves our claim.
 The difference between toroidal theta and wild groups appears here as the dichotomy
 whether the series $\sum_{\sigma\neq 0} \eta^{(\sigma)}$ converges (theta case, see \cite[Prop. 2.2.5]{abe-kopferman01})
 or not (wild case, see \cite[Prop. 2.2.7]{abe-kopferman01}).
 
 \medskip
 The identity \eqref{eq:quotient01} follows now precisely as \cite[Theorem 2.2.6]{abe-kopferman01}:
 just use our claim above instead of \cite[Prop. 2.2.5]{abe-kopferman01}.
 \end{proof}

\section{Nilmanifolds foliated in toroidal groups }\label{sect: main}
We now show how in some cases Dolbeault cohomology of a nilmanifold with left-invariant complex structure can be computed even if it not fibred but only foliated.
\subsection{Setting and  main theorem}

 Let $G$ be a simply-connected real nilpotent Lie group with Lie algebra $\lg$  and with  an integrable almost complex structure $J\colon \lg\to \lg$.  Assume that $G$ contains a lattice $\Gamma$ and denote by $M=\Gamma\backslash G$ the corresponding complex  nilmanifold  with invariant complex structure. 
 
Assume $\gothf\subset \lg$ is an abelian $J$-invariant ideal. We denote the corresponding closed subgroup of $G$ again with $\gothf\subset G$ (since the exponential map preserves the group structure) and let $p \colon G\to H = \gothf\backslash G$ be the natural projection and $\lh$ be the Lie algebra of $H$, which carries an induced complex structure.

The projection $p$ induces a fibration on $M$ if and only if  $\Gamma_\gothf = \Gamma\cap \gothf$ is a cocompact lattice in $\gothf$; otherwise the translates of $F = \Gamma_\gothf\backslash\gothf$ are the leaves of a holomorphic foliation on $M$.

We assume the following:
\begin{enumerate}
 \item There is a subgroup $\Gamma_\gothf \subset \Gamma' \subset \Gamma$ such that $\Gamma_H = p(\Gamma')$ is a discrete cocompact subgroup of $H$.
 
 We denote the quotient by $B' = \Gamma_H\backslash H$ and set $\Lambda = \Gamma/\Gamma'$.
 \item $F = \Gamma_\gothf\backslash \gothf$ is a non-compact abelian complex Lie group.\footnote{Note that $F$ is compact if and only if $p$ induces a fibration on $M$ if and only if $\gothf$ is a $\Gamma$-rational subspace of $\lg$. This case is treated in \cite{con-fin01}.}  
 \end{enumerate}
We then get a diagram
\begin{equation}\label{diag: toroidal foliation}
\begin{tikzcd}
   \gothf \rar\arrow{dr} & F \arrow{dr} \arrow[equal]{r}& F 
   \arrow{dr}\\
    & G\dar{p}\rar{\tilde\pi} &M'\dar{p'} \rar{f}[swap]{/\Lambda} &M\\
    & H \rar& B'  \\
  \end{tikzcd} 
\end{equation}
with the following properties:
\begin{enumerate}
\item $p'\colon M':=\Gamma'\backslash G \to B'$ is a locally trivial, holomorphic submersion with fibre $F$ over a compact complex nilmanifold; $p'$ is not proper.
\item The translates of $F = \Gamma_\gothf\backslash\gothf$, i.\,e., the images of  fibres of $p'$, are the leaves of a holomorphic foliation on $M$ with leaf-space $B=\Lambda\backslash B'$. 
  \item On $M'$ we have the residual action of $\Lambda = \Gamma/\Gamma'$ and the quotient map $f$ is unramified.
\end{enumerate}

In such a situation one can compute the Dolbeault cohomology of $M$ 
by a \emph{composition} of two spectral sequences, namely one can first compute the Dolbeault cohomology of $M'$ via the Leray spectral sequence
\begin{equation}\label{eq: Leray for p'}H^r(B', R^sp'_*\Omega_{M'}^p) \implies H^{r+s}(\Omega^p_{M'}),
   \end{equation}
and then use the result in the equivariant sheaf cohomology spectral sequence for $f$ \cite[subsection 5.2]{Grothendieck-Tohoku} 

 \begin{equation}\label{eq: equivariant cohomology spectral sequence} H^r(\Lambda, H^s(M', \Omega_{M'}^p))\implies H^{r+s}_\Lambda (M', \Omega_{M'}^p)\isom H^{r+s}(M, \Omega_M^p) = H^{p, r+s}(M).
 \end{equation}
The idea of the following result is to compare these to analogous spectral sequences on the level of left-invariant forms, that is,
from Lie algebra cohomology.
 
Let $\gothf_0 \subset \Gamma_\gothf\tensor \IR \subset \gothf$ be the largest complex subspace of the real subspace spanned by $\Gamma_\gothf$ in $\gothf$. 
\begin{thm}\label{thm: toroidal foliation}
In the above situation we choose toroidal coordinates $\gothf=\gothk \oplus {\gothf_0}$
on $F$ (in the sense of Section \ref{sect: toroidal groups}).

Assume that Conjecture \ref{conj} holds for the complex nilmanifold $B'$ and that there is a $J$-invariant subalgebra $\lg_0^{0,1}\subset \lg^{0,1}$ fitting in the exact diagram
\begin{equation}\label{diag: lie algebras}
 \begin{tikzcd}
  {} & 0\dar & 0\dar\\
  0\rar& \nulleins{\gothf_0}\rar\dar & \nulleins{\gothg_0}\rar\dar & \nulleins\lh \dar[equal]\rar & 0\\
   0\rar& \nulleins\gothf\rar\dar{q} & \nulleins\gothg\rar \dar& \nulleins\lh \rar & 0\\
   & \nulleins \gothk \dar\rar[equal] & \nulleins \gothk \dar\\&0&0
 \end{tikzcd}
\end{equation}
Assume further that $F$ is a toroidal group.
Then Conjecture \ref{conj} holds for $M$.
\end{thm}

\begin{rem}
 The existence of such a $\gothg_0$ as in the theorem is needed in the proof, but we do not know of an example where it does not exist. An example, where this is quite natural, can be found in Section \ref{sect: h7}.
\end{rem}

\subsection{Proof of  Theorem \ref{thm: toroidal foliation}}
The proof of the theorem consists of comparing two geometric spectral sequences, namely for the one for $p'$ and for $f$, with two appropriately set up Hochschild-Serre-spectral sequences in Lie algebra cohomology. We continue to use the notations from the previous section. 

Note that we can always replace $\Gamma$ by a sublattice of finite index by \cite[Prop. 3.16]{rollenske10}, so we may assume that $\Lambda$ is torsionfree.

\subsubsection{The case $F$ a toroidal theta group}
We first treat the case where the typical leaf of the foliation $F$ is a toroidal theta group, which is more transparent.

\begin{lem}\label{lem: fibrewise cohomolgies theta}
If $F$ is a toroidal theta group then the inclusion on the level of forms induces an isomorphism 
 $H^s(\nulleins{\gothf_0},  \Wedge^p\einsnull{\lg^*})\isom H^{s}(F, \Omega_{M'}^p|_F)$.
\end{lem}
\begin{proof}
 The exact sequence of locally free sheaves
 \[ 0\to \Omega_F\to \Omega_{M'}|_F\to N_{F/M'}\to 0\]
induces a filtration on the sheaf $\Omega^p_{M'}|_F$, which corresponds to a filtration on the $\nulleins\gothk$-module $\Wedge^p\einsnull{\lg^*}$. Instead of comparing the cohomologies directly, we compare the $E_2$-terms of the resulting spectral sequences. These are isomorphic by Theorem \ref{thm: vogt}, so the inclusion on the level of forms also induces an isomorphism on the limits, which yields the claim.
\end{proof}

\begin{lem}\label{lem: spectral sequences for p' toroidal theta}
If $F$ is a toroidal theta group then the inclusion
 \[ \Wedge^*\nulleins{\lg_0^*}\tensor \Wedge^p\einsnull{\lg^*}\into \ka^{p,*}(M')\]
 induces an isomorphism
 \[ H^q(\nulleins{\lg_0}, \Wedge^p\einsnull{\lg^*}) \isom {H^{p,q}(M')},\]
 which is compatible with the natural $\Lambda$-action on both sides.
 
 In particular, $H^{p,q}(M')$ is finite-dimensional.
\end{lem}
\begin{proof}
 The fibration $p'\colon M'\to B'$ defines a filtration on the complex on the right hand side, which via the inclusion is compatible with the filtration defined on the left hand side by the first row of \eqref{diag: lie algebras}.
 This gives a map on the corresponding $E_2$ spectral sequences, which are the Hochschild-Serre-spectral sequence
 \begin{equation}\label{eq: Hochschild Serre for p'}
H^r(\nulleins{\lh}, H^s(\nulleins{\gothf_0},  \Wedge^p\einsnull{\lg^*})) \implies H^{r+s} ( \nulleins{\lg_0},  \Wedge^p\einsnull{\lg^*}),  \end{equation}  

  and the Leray spectral sequence of $p'$ \eqref{eq: Leray for p'}.

By Lemma  \ref{lem: fibrewise cohomolgies theta} the map is an isomorphism on $E_2$ and since  both spectral sequences are compatible with the $\Lambda$-action, our claim follows.
\end{proof}

 Now we consider again two spectral sequences. First the equivariant sheaf cohomology sequence \eqref{eq: equivariant cohomology spectral sequence}
 and second the Hochschild-Serre spectral sequence induced by the middle column of diagram \eqref{diag: lie algebras},
\begin{equation}\label{eq: Hochschild Serre for f} H^r(\nulleins{\gothk}, H^s(\nulleins{\lg_0}, \Wedge^p\einsnull{\lg^*})) \implies H^{r+s}(\nulleins{\lg}, \Wedge^p\einsnull{\lg^*}) = H^{p, r+s}(\lg, J).
 \end{equation}
The theorem now follows from the next Lemma: 
 \begin{lem}\label{lem: spectral sequences for f toroidal theta}
 If $F$ is a toroidal theta group then there is a morphism of spectral sequences from \eqref{eq: Hochschild Serre for f} to  \eqref{eq: equivariant cohomology spectral sequence} which induces an isomorphism on the $E_2$-term, and hence on the $E_\infty$-term.
 \end{lem}
\begin{proof}
By Lemma \ref{lem: spectral sequences for p' toroidal theta} the inclusion map on forms induces an isomorphim on the coefficients of the spectral sequences. 

As mentioned above, we may assume that $\Lambda$ does not have torsion. Then by construction we have $\Lambda\tensor \IC \isom \gothk$ and the action of $\Lambda$ and of $\gothk$ is induced by the coadjoint action.
 \end{proof}

 \subsubsection{The case where $F$ is not toroidal theta}
In this case the Dolbeault cohomology of the fibre is not finite-dimensional, which we have to deal with by carefully taking Hausdorff quotients.

 \begin{lem}\label{lem: fibrewise cohomolgies wild}
If $F$ is a toroidal wild group then the inclusion on the level of forms induces a natural splitting 
 $ H^{s}(F, \Omega^p_{M'}|_F)\isom H^s(\nulleins{\gothf_0},  \Wedge^p\einsnull{\lg^*})\oplus \overline 0$, 
and thus a natural isomorphism $\overline H^{s}(F, \Omega^p_{M'}|_F)\isom H^s(\nulleins{\gothf_0},  \Wedge^p\einsnull{\lg^*})$.
\end{lem}
\begin{proof}
We want to  argue exacly as in the proof of Lemma \ref{lem: fibrewise cohomolgies theta}, however we need to apply Proposition \ref{prop: Hausdorff quotients of Dolbeault cohomology for abelian Lie groups} instead of Theorem \ref{thm: vogt} if $F$ is toroidal wild. 

But in every step Lemma \ref{lem: Hausdorff quotient for group cohomology} applies, because $F$ is an Eilenberg-MacLane space and all bundles involved are flat, so the cohomologies we compute can all be reinterpreted as group cohomologies for $\pi_1(F) = \Gamma_\gothf$. 
\end{proof}

\begin{lem}\label{lem: spectral sequences for p' toroidal wild}
If $F$ is a toroidal group then the inclusion
 \[ \Wedge^*\nulleins{\lg_0^*}\tensor \Wedge^p\einsnull{\lg^*}\into \ka^{p,*}(M')\]
 induces an isomorphism
 \[ H^q(\nulleins{\lg_0}, \Wedge^p\einsnull{\lg^*}) \isom \overline{H^{p,q}(M')},\]
 which is compatible with the natural $\Lambda$-action on both sides.
 
 In particular, $\overline{H^{p,q}(M')}$ is finite-dimensional.
\end{lem}
\begin{proof}
 Dolbeault cohomology of $M'$ is computed via the spectral sequence \eqref{eq: Leray for p'}. Since the fibration is locally trivial the coefficient sheaf is a local system with fibre $H^s(F, \Omega^p_{M'}|_F)$. 
 Since $B'$ is a nilmanifold, it is an Eilenberg--MacLane space $K(\Gamma_H, 1)$ and thus the cohomology of a local system on $B'$ is, by definition, the group cohomology of $\Gamma_H$ with values in the module $H^s(F, \Omega^p_{M'}|_F)$.

By Lemma \ref{lem: fibrewise cohomolgies wild}   the inclusion of left-invariant forms into $H^s(F, \Omega^p_{M'}|_F)$ satisfies the assumptions of Lemma \ref{lem: Hausdorff quotient for group cohomology} and thus the map on the $E_2$ terms \eqref{eq: Hochschild Serre for p'} to the $E_2$-term of \eqref{eq: Leray for p'} is an isomorphism after taking the Hausdorff quotient of the coefficients and we conclude by Lemma \ref{lem: Hausdorff quotient for group cohomology}.
\end{proof}

The rest of the proof proceeds as in the toroidal theta case, again making use of Lemma \ref{lem: Hausdorff quotient for group cohomology}.

\section{Dolbeault cohomology of comlex nilmanifolds of real dimension six}
To show that the result from Section \ref{sect: main} can be applied in practice we prove the following: 
\begin{thm}\label{thm: dim 6}
 Let $M$ be a nilmanifold of real dimension $6$ and let $J$ be a left-invariant complex structure on $M$. Then Conjecture \ref{conj} holds for $(M,J)$.
\end{thm}
Before we can go into the proof we need to study one particular six-dimenional real Lie algebra in detail. 
\subsection{The structure of $\lh_7$ and related Lie algebras}\label{sect: h7}
In this section we study the real $6$-dimensional Lie algebra $\lh_7$.
\begin{lem}
 Let $V$ be a real vector-space and $\lg = V\oplus \Wedge^2 V$ the free 2-step nilpotent Lie algebra on $V$. 
 \begin{enumerate}
  \item All Lie algebra automorphisms of $\lg$ are of the form 
\[ \begin{pmatrix}
    \phi & 0\\\psi & \Wedge^2\phi
   \end{pmatrix}
   \]
with $\phi\in \Gl(V)$ and $\psi \in \Hom (V, \Wedge^2 V)$. 
\item Up to Lie algebra automorphism there exists a unique rational structure for $\lg$.
 \end{enumerate}
\end{lem}
\begin{proof}
 The first item is clear. For the second item, suppose $\lg_\IQ$ is a $\IQ$-structure. Choose as basis $w_1, \dots, w_m$ for $\lg_\IQ$, such that $w_1, \dots, w_d$ span a complement of the commutator. Then there is an automorphism of $\lg$ sending $w_1, \dots, w_d$ to a fixed basis $\kb$ of $V$, which transform $\lg_\IQ$ in the $\IQ$-structure spanned by $\kb$ and $\Wedge^2\kb$.
\end{proof}

\begin{prop}\label{prop: structure h7} Consider the Lie algebra $\lh_7=(0,0,0,12,13,23)\isom\IR^3\oplus \Wedge^2\IR^3$. Then 
\begin{enumerate}
 \item $\lh_7$ has a unique rational structure up to Lie algebra automorphisms.
 \item $\lh_7$ has a unique complex structure $J_0$ up to Lie algebra automorphisms with a representative given by $J_0e_{2i-1} = e_{2i}, i =1,2,3.$
 \item  The only non-trivial, $J_0$-invariant ideals in $\lh_7$ are $\gothf= \langle e_3, \dots, e_6\rangle$ and $\gotha= \langle e_5, e_6\rangle$.
 \end{enumerate}
\end{prop}
\begin{proof}
The first item is a special case of the previous Lemma. The second item is a simple computation or follows also quickly from the normal forms given in \cite{ugarte07}.

For the last item assume $\gothb$ is an even-dimensional, non-trivial ideal of $\lh_7$ and write it as a direct sum of vector spaces $\gothb = \gothb_0\oplus \gothb_1$, where $\gothb_1 = \gothb\cap [\lh_7, \lh_7]$. Then $[\lh_7, \gothb_0]\subset \gothb_1$, so using the structure equations of $\lh_7$ we get that either $\gothb= \gothb_1$ is contained in the commutator or $\dim_\IR\gothb = 4$ and $\gothb$ contains the commutator.
 If $\gothb$ is further assumed to be $J_0$-invariant, then by the above we have $\gothb = \gotha$ in the former case and $\gothb = \gothf$ in the latter, because  $\gotha$ is the only $J_0$-invariant subspace contained in the commuator of $\lh_7$, and  $\gothf$ is the smallest such subspace containing the commutator. 
 \end{proof}

 Consider $H_7$ together with the left-invariant complex structure $J_0$ introduced above. Then $\gothf = \langle e_3, e_4, e_5, e_6\rangle$ is an abelian, $J$-invariant ideal of $\lh_7$ with abelian quotient and as such induces a fibre bundle structure 
 \begin{equation}\label{eq: fibration H7}
\begin{tikzcd} \gothf = \IC^2 \rar[hookrightarrow] & H_7\dar \\ & \IC\end{tikzcd}
 \end{equation}
 with complex coordinates $dz_1 = e^3+\I e^4$, $dz_2 = e^5+\I e^6$ on the fibre.

Assume $\Gamma\subset H_7$ is a lattice and let $M = (\Gamma \backslash H_7, J_0)$. Recall that by \cite[Lemma 5.1.4, Theorem 5.1.11]{Cor-Green} (compare \cite[Sect. 3.2.1]{rollenske11}) for an ideal $\gothb\subset \lh_7$ corresponding to a  Lie  subgroup $B\subset H_7$ the following are equivalent:
\begin{enumerate}
 \item $B\cap \Gamma$ is a lattice in $B$ and $\gothb$ is  $J_0$-invariant.
 \item $B$ maps to a compact complex submanifold in $M$.
 \item The fibration $H_7 \to B\backslash H_7$ induces on $M$ a holomorphic submersion onto a complex nilmanifold.
 \item The ideal $\gothb$ is $J_0$ invariant and $\Gamma$-rational, that is, \[\left(\langle \log\Gamma\rangle_\IQ \cap \gothb\right)\tensor \IR \isom \gothb. \]
\end{enumerate}
With this, we can understand the geometry of $M$ quite well.

\begin{prop}\label{prop: H7 fibred or foliated by toroidal groups}
Let $\Gamma\subset H_7$ be a lattice and let $M = (\Gamma\backslash H_7 , J_0)$. If $\lh_7$ contains a non-trivial $J_0$-invariant and $\Gamma$-rational ideal, then $\gothf$ is also $J_0$-invariant and $\Gamma$-rational and \eqref{eq: fibration H7} induces on $M$ a (non-principal) 2-torus bundle over an elliptic curve. 

If $\gothf$ is not $\Gamma$-rational then the  complex abelian Lie group $F = \Gamma\cap \gothf\backslash\gothf$ is toroidal and both the toroidal theta and the toroidal wild case occur for specific choices of $\Gamma$. 
\end{prop}
\begin{proof}
By Proposition \ref{prop: structure h7} there are only two non-trivial, $J_0$-invariant ideals, so we only need to show that if ${\gotha} = \langle e_5, e_6\rangle$ is $\Gamma$-rational, then so is $\gothf$. 

So assume $\gotha$ is $\Gamma$-rational and let $A$ be the corresponding Lie subgroup, which induces a principal fibration in elliptic curves on $M$. Then $\gothf/\gotha$ is the centre of $\lh_7/\gotha$ and thus a $\Gamma\cap A \backslash\Gamma$-rational subspace    \cite[Sect.\ 5]{Cor-Green}. As $\gotha$ is $\Gamma$-rational, the preimage $\gothf$ is $\Gamma$-rational as well.

If $\gothf$ is not $\Gamma$-rational then $F$ does not contain compact complex subgroups, because the only $J_0$-invariant subspace contained in the real span of the lattice is $\gotha$, which is not $\Gamma$-rational by the above. So by Proposition \ref{prop: structure abelian cx Lie-groups} $F$ is a non-compact toroidal group.

Examples where $F$ is toroidal wild repectively toroidal theta are constructed in Example \ref{ex: H7 theta and wild} below. 
\end{proof}

\begin{exam}\label{ex: H7 theta and wild}
 Let $a\in \IR$ and consider in $\lh_7$ the vectors
 \[ v_1 = \sqrt 2 e_1, \, v_2 = \sqrt 2 e_2,\,  v_3 = \sqrt2 (e_3 + ae_1), v_4 = e_4,\, v_5= e_5,\, v_6 = ae_4 - e_6.\]
The $\IZ$-span $V_\IZ =\langle v_1, \dots, v_6\rangle_\IZ$  is a Lie-subring and by the Baker--Campbell--Hausdorff-formula the exponential of this Lie-subring is a lattice $\Gamma_a$ in $H_7$. 

Computing the intersection of $V_\IZ\cap \gothf$ we see that $\gothf$ is $\Gamma_a$-rational if and only if $a \in \IQ$. 

Now assume $a\not \in \IQ$ and  consider $\gothf$ as a complex $2$-dimensional vector space with basis $\tilde z_1 = -\I z_1$ and  $z_2$. Then we obtain the period matrix for $F$ as in \eqref{eq: period matrix} by writing $v_4, v_5, v_6$, the generators of  $\Gamma_\gothf = V_\IZ\cap \gothf$, in complex coordinates, which gives
\[ \begin{pmatrix}
   1  & 0 & a\\
   0 & 1& \I
   \end{pmatrix}.
\]
As explained in  Example  \ref{ex: theta and wild in dim two}, the group $F$ is toroidal, because $a\not \in \IQ$ and both the toroidal wild and the toroidal theta case can occur for specific choices of $a$. 
\end{exam}

\subsection{Proof of Theorem \ref{thm: dim 6}}

Let $\lg$ be the Lie algebra of the universal cover $G$ of $M = \Gamma\backslash G$. 
 Nilpotent Lie algebras of real dimension $6$ admitting a complex structure have been classified by Salamon \cite{salamon01}. It was shown in \cite[Proof of Thm.\,B]{rollenske09b} that Conjecture \ref{conj} holds independent of the choice of complex structure and lattice unless $\lg \isom \lh_7$.
 
Thus it remains to treat the case $\lg = \lh_7$ and we continue to use the notations introduced in Section \ref{sect: h7}. 
By Proposition \ref{prop: structure h7} we may assume that $J=J_0$.

If $\gothf\subset \lh_7$ is $\Gamma$-rational then by Proposition \ref{prop: H7 fibred or foliated by toroidal groups} the fibration \eqref{eq: fibration H7} induces a holomorphic 2-torus fibration over an elliptic curve on $M$ and Conjecture \ref{conj} holds by \cite{con-fin01}.

If $\gothf$ is not $\Gamma$-rational then again by Proposition \ref{prop: H7 fibred or foliated by toroidal groups} $M$ is foliated in toroidal groups $F$. 
Consider the basis  of $\einsnull{\lh_7}$ given by  $X_1 = e_1+\I e_2, X_2 = e_3+\I e_4, X_3 = e_5+ \I e_6$. Then Theorem \ref{thm: toroidal foliation} applies with $\gothf_0 = \gotha = \langle e_5, e_6\rangle$ and the choice $\nulleins{\lg_0} = \langle \bar X_1, \bar X_3\rangle$.\footnote{It is interesting to observe that $\einsnull{\lg_0}\oplus \overline{\einsnull{\lg_0}}$ is not a subalgebra of ${\lh_7}_{\IC}$.}
Thus Conjecture \ref{conj} holds also in this case, which concludes the proof. \qed


\end{document}